\documentclass[]{class2}
\usepackage{graphicx, xfrac, lineno, float, subcaption, tasks, comment, xcolor, booktabs, multirow}
\usepackage[normalem]{ulem}
\usepackage{datetime}
\usepackage{colortbl}
\usepackage{enumerate}

\settasks{
	counter-format=(tsk[r]),
	label-width=4ex
}

\begin{document}

%\linenumbers

\begin{frontmatter}

\titledata{Accordion graphs: Hamiltonicity, matchings and isomorphism with quartic circulants}{Dedicated to Professor Anton Buhagiar (1954--2020) --- a brilliant mathematician, a passionate musician, a dedicated teacher, a cheerful friend.}           % title of the paper

\authordata{John Baptist Gauci}
{Department of Mathematics, University of Malta, Malta}
{john-baptist.gauci@um.edu.mt}
{}

\authordata{Jean Paul Zerafa}
{Dipartimento di Scienze Fisiche, Informatiche e Matematiche\\ Universit\`{a} degli Studi di Modena e Reggio Emilia, Italy;\\
Department of Technology and Entrepreneurship Education, University of Malta, Malta;\\
Department of Computer Science, Faculty of Mathematics, Physics and Informatics\\ Comenius University, Mlynsk\'{a} Dolina, 842 48 Bratislava, Slovakia}
{zerafa.jp@gmail.com}
{}

\keywords{Hamiltonian cycle, perfect matching, quartic graph, circulant graph.}
\msc{05C70, 05C45, 05C60.}

\begin{abstract}
Let $G$ be a graph of even order and let $K_{G}$ be the complete graph on the same vertex set of $G$. A pairing of a graph $G$ is a perfect matching of the graph $K_{G}$. A graph $G$ has the Pairing-Hamiltonian property (for short, the PH-property) if for each one of its pairings, there exists a perfect matching of $G$ such that the union of the two gives rise to a Hamiltonian cycle of $K_G$. In 2015, Alahmadi \emph{et al.} gave a complete characterisation of the cubic graphs having the PH-property. Most naturally, the next step is to characterise the quartic graphs that have the PH-property. In this work we propose a class of quartic graphs on two parameters, $n$ and $k$, which we call the class of accordion graphs $A[n,k]$. We show that an infinite family of quartic graphs (which are also circulant) that Alahmadi \emph{et al.} stated to have the PH-property are, in fact, members of this general class of accordion graphs. We also study the PH-property of this class of accordion graphs, at times considering the pairings of $G$ which are also perfect matchings of $G$. Furthermore, there is a close relationship between accordion graphs and the Cartesian product of two cycles. Motivated by a recent work by Bogdanowicz (2015), we give a complete characterisation of those accordion graphs that are circulant graphs. In fact, we show that $A[n,k]$ is not circulant if and only if both $n$ and $k$ are even, such that $k\geq 4$.
\end{abstract}

\end{frontmatter}
\section{Introduction}
The graphs considered in this paper are undirected and simple (without loops or multiple edges). A graph $G$ having vertex set $V(G)$ of cardinality $n$ and edge set $E(G)$ has a \emph{perfect matching} if there is a set of exactly $\frac{n}{2}$ edges of $G$ which have all the $n$ vertices of $G$ as endvertices. This tacitly implies that a necessary condition for a graph to have a perfect matching is that it has an even number of vertices. Let $k$ be a non-negative integer. A \emph{path} on $k$ vertices, denoted by $P_{k}$, is a sequence of pairwise distinct vertices $v_1,\ldots,v_{k}$ with corresponding edge set $\{v_{i}v_{i+1}:i\in \{1,\ldots,k-1\}\}$ (if $k>1$). For $k\geq 3$, a \emph{cycle} of length $k$ (or a $k$-cycle), denoted by $C_{k}=(v_{1}, \ldots, v_{k})$, is a sequence of mutually distinct vertices $v_1,v_2,\ldots,v_{k}$ with corresponding edge set $\{v_{1}v_{2}, \ldots, v_{k-1}v_{k}, v_{k}v_{1}\}$. A \emph{Hamiltonian cycle (path)} of a graph $G$ is a cycle (path) in $G$ that visits all the vertices of $G$ exactly once. We refer the reader to \cite{Diestel} for definitions and notation not explicitly stated here.

The two main concepts of this paper, namely Hamiltonicity and perfect matchings, have been extensively studied in literature by various authors and in various settings. The study of both together has recently been put in the limelight in \cite{ThomassenEtAl}, where the authors ask whether a perfect matching of the complete graph $K_G$ on the same vertex set as a graph $G$, called a \emph{pairing} of $G$, can be \emph{extended} to a Hamiltonian cycle of $K_G$ by using only edges of $G$. Equivalently, given a pairing $M$ of $G$, we say that $M$ can be extended to a Hamiltonian cycle $H$ of $K_G$ if we can find another perfect matching $N$ of $G$ such that $M\cup N=E(H)$, where $E(H)$ is the corresponding edge set of the Hamiltonian cycle $H$. By using this terminology, the authors of \cite{ThomassenEtAl} define a graph $G$ on an even number of vertices as having the \emph{Pairing-Hamiltonian property} (or, for short, the \emph{PH-property}) if every pairing $M$ of $G$ can be extended to a Hamiltonian cycle $H$ of $K_G$. On a similar flavour, the authors of \cite{PMHAbreuEtAl} define the \emph{Perfect-Matching-Hamiltonian property} (or, for short, the \emph{PMH-property}) for a graph $G$ if every perfect matching of $G$ can be extended to a Hamiltonian cycle of $K_{G}$, which in this case, would also be a Hamiltonian cycle of $G$ itself. For simplicity, if a graph $G$ admits the PMH-property we shall say that $G$ is \emph{PMH}. It can be easily seen that if a graph does not have the PMH-property, then it surely does not have the PH-property, although the converse is not true. We remark that the PMH-property was already studied in the 1970s by H\"{a}ggkvist and Las Vergnas (see \cite{haggkvist, lasvergnas}).

A complete characterisation of the cubic graphs having the PH-property was given in \cite{ThomassenEtAl}, and thus the most obvious next step would be to characterise $4$-regular graphs which have the PH-property. This endeavour proved to be more elusive, and thus far a complete characterisation of such quartic graphs remains unknown. In an attempt to advance in this direction, in Section \ref{Sect-AccGphs}, we define a class of graphs on two parameters, $n$ and $k$, which we call the class of \emph{accordion graphs} $A[n,k]$. This class presents a natural generalisation of the well-known antiprism graphs, and of a class of graphs which is known to have the PH-property, corresponding to $A[n,1]$ and $A[n,2]$, respectively. In the same section, we discuss some fundamental properties and characteristics of accordion graphs and, in particular, we see that accordion graphs can be drawn in a grid-like manner, which resembles a drawing of the Cartesian product of two cycles $C_{n_{1}}\square C_{n_2}$, for appropriate cycle lengths $n_{1}$ and $n_{2}$. In Section \ref{Sect-PH} we prove that all antiprism graphs have the PMH-property but only four of them also have the PH-property. In the same section, we provide a proof that $A[n,2]$ has the PH-property, which result, although known to be communicated by the authors of \cite{ThomassenEtAl}, has no published proof. These encouraging outcomes motivate our proposal of the class of accordion graphs as a possible candidate for graphs having the PMH-property and/or the PH-property. Empirical evidence suggests that, apart from the above mentioned, there are (possibly an infinite number of) other accordion graphs which have the PMH-property, and possibly some of them even have the PH-property, but a proof for this is currently unavailable. In Section \ref{Sect-NotPMH}, by extending an argument introduced in \cite{CpCq}, we show that we can exclude some graphs $A[n,k]$ from this search for graphs having the PMH- and/or the PH-property. In fact, we prove that the graphs $A[n,k]$ for which the greatest common divisor of $n$ and $k$ is at least $5$ do not have the PMH-property. The technique used does not seem to lend itself when coming to show whether the remaining accordion graphs have, or do not have, the PH-property or the PMH-property. 

In 2015, Bogdanowicz \cite{CpCqCirculant} gave all possible values of $n_{1}$ and $n_{2}$ for which the graph $C_{n_{1}}\square C_{n_2}$ is a circulant graph, namely when $\gcd(n_{1},n_{2})=1$. Due to the similarity between the two classes of graphs, in Section \ref{Sect-Circulant}, we give a complete characterisation of which accordion graphs are circulant graphs. We finally pose some related questions and open problems in Section \ref{Sect-Concl}.

\section{Accordion graphs} \label{Sect-AccGphs}

In the sequel, operations (including addition and subtraction) in the indices of the vertices $u_{i}$ and $v_{i}$ in an accordion graph (as in the following definition) are taken modulo $n$, with complete residue system $[n]=\{1,\ldots, n\}$. For simplicity, and unless there is room for confusion, we shall omit writing $\pmod{n}$ when referring to the indices of $u_i$ and $v_i$.

\begin{definition} \label{def AccordionGraphs}
Let $n$ and $k$ be integers such that  $n\geq 3$ and $0<k\le\frac{n}{2}$. The \emph{accordion graph} $A[n,k]$  is the quartic graph with vertices $\{u_1, u_2,\ldots, u_{n}, v_1, v_2,\ldots , v_{n}\}$ such that the edge set consists of the edges
\begin{linenomath}
$$\{u_{i}u_{i+1},v_{i}v_{i+1},u_iv_i,u_iv_{i+k}: i\in[n]\}.$$
\end{linenomath}
The edges $u_{i}u_{i+1}$ and $v_{i}v_{i+1}$ are called the \emph{outer-cycle edges} and the \emph{inner-cycle edges}, respectively, or simply the \emph{cycle edges}, collectively; and the edges $u_iv_i$ and $u_iv_{i+k}$ are called the \emph{vertical spokes} and the \emph{diagonal spokes}, respectively, or simply the \emph{spokes}, collectively. For simplicity, we sometimes refer to the accordion graph $A[n,k]$ as the accordion $A[n,k]$.
\end{definition}

\begin{figure}[h]
\centering
\includegraphics[width=.603\textwidth]{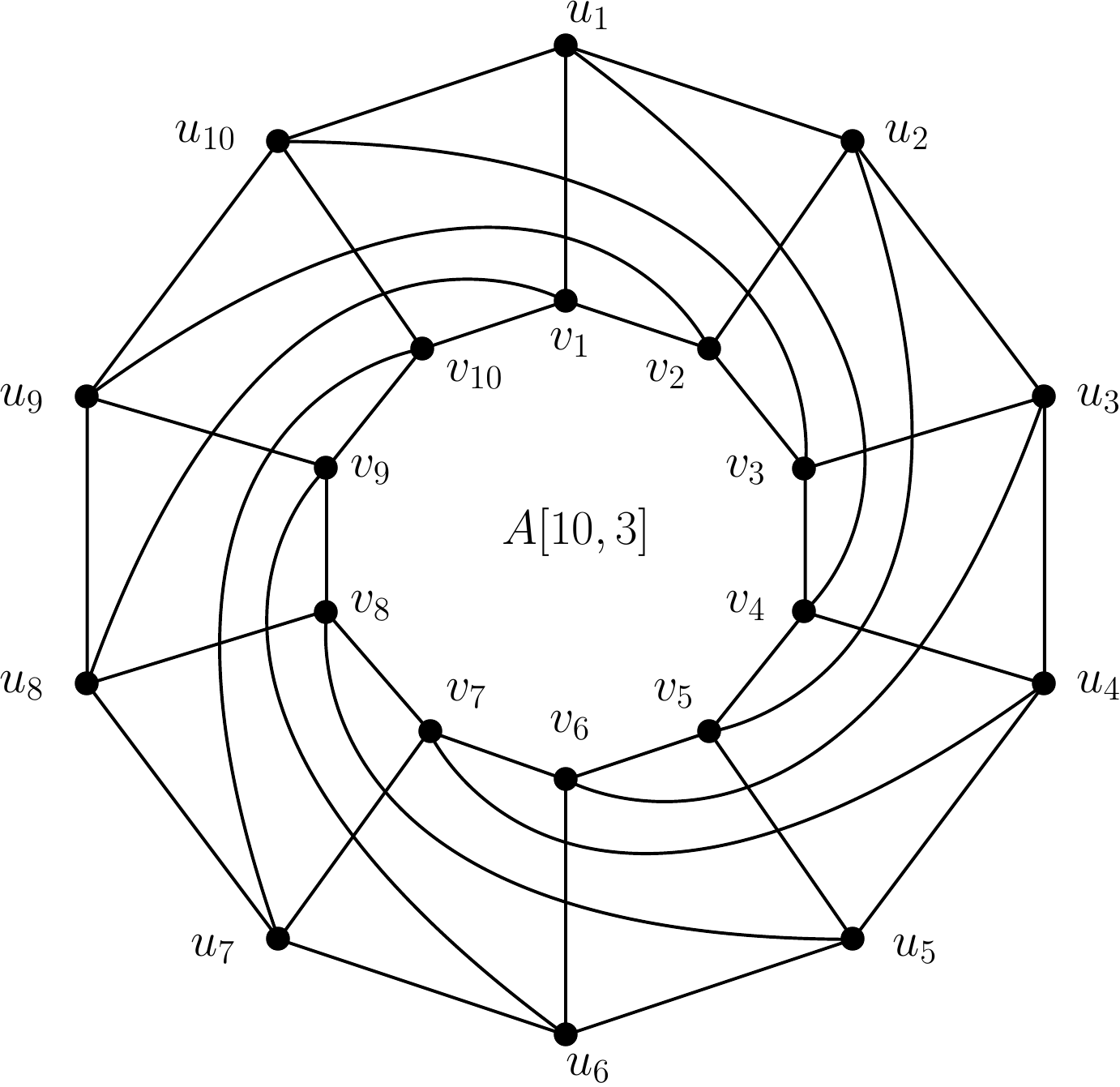}
\caption{The accordion graph $A[10,3]$}
\end{figure}

An observation that will prove to be useful in the sequel revolves around the greatest common divisor of $n$ and $k$, denoted by $\gcd(n,k)$. However, before proceeding further, we give the definition of the Cartesian product of graphs. The \emph{Cartesian product} $G\square H$ of two graphs $G$ and $H$ is a graph whose vertex set is the Cartesian product $V(G) \times V(H)$ of $V(G)$ and $V(H)$. Two vertices $(u_i,v_j)$ and $(u_k,v_l)$ are adjacent precisely if  $u_i=u_k$ and $v_jv_l\in E(H)$ or $u_iu_k \in E(G)$ and $v_j=v_l$.

\begin{remark}\label{RemarkDrawing}
The graph obtained from $A[n,k]$ by deleting the edges
%\begin{linenomath}
$$\{u_{tq}u_{tq+1}, v_{tq}v_{tq+1} :  q=\gcd(n,k)  \textrm{ and }  t\in\{1,\ldots,\tfrac{n}{q}\}\}$$
%\end{linenomath}
is isomorphic to the Cartesian product $C_{\frac{2n}{q}} \square P_{q}$. This can be easily deduced by an appropriate drawing of $A[n,k]$, as shown in Figure \ref{Figure AccCart} for the case when $k=5$ and $\gcd(n,k)=5$. Thus, any perfect matching of $C_{\frac{2n}{q}} \square P_{q}$ is also a perfect matching of $A[n,k]$, although the converse is trivially not true.
\end{remark}

\begin{figure}[h]
\centering
\includegraphics[width=.7\textwidth]{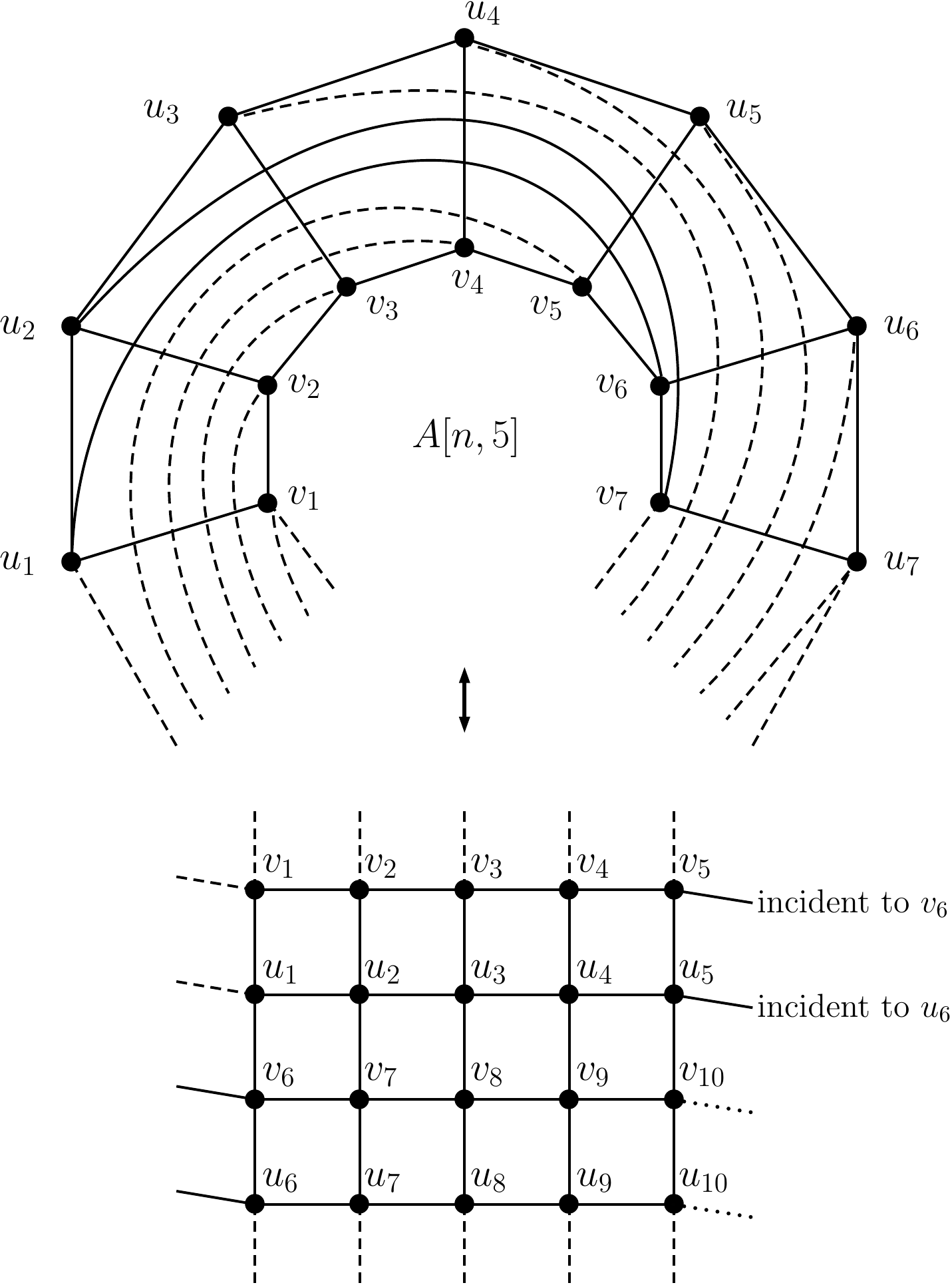}
\caption{Two different drawings of $A[n,k]$ when $k=5$ and $\gcd(n,k)=5$}
\label{Figure AccCart}
\end{figure}

\section{The accordion graph $A[n,k]$ when $k\leq 2$} \label{Sect-PH}

\subsection{$A[n,1]$}
As already mentioned above,  the accordion graph $A[n,1]$ is isomorphic to the widely known antiprism graph $A_n$ on $2n$ vertices.  Let $M$ be a perfect matching of $A_{n}$. We note that, if $M$ contains at least one vertical spoke, then no diagonal spoke can be contained in $M$, and if any inner-cycle edges are in $M$, then the outer-cycle edges having the same indices must also belong to $M$. A similar argument can be made if $M$ contains diagonal spokes. Thus, for every $i,j\in[n]$, $u_{i}v_{i}\in M$ or $\{u_{i}u_{i+1},v_{i}v_{i+1}\}\subset M$ if and only if  $u_{j}v_{j+1}\not\in M$ or $\{u_{j}u_{j+1}, v_{j+1}v_{j+2}\}\not\subset M$. This can be summarised in the following remark.

\begin{remark}\label{Remark PM in An}
Let $M$ be a perfect matching of $A_{n}$. Then, $M$ is either a perfect matching of $A_{n}-\{u_{i}v_{i}:i\in[n]\}$ or of $A_{n}-\{u_{i}v_{i+1}:i\in[n]\}$.
\end{remark}

Consequently, in what follows, without loss of generality, we only consider perfect matchings of $A_{n}$ containing spokes of the type $u_{i}v_{i}$, that is, vertical spokes, and if a perfect matching contains the edge $u_{j}u_{j+1}$, then it must also contain the edge $v_{j}v_{j+1}$.

\begin{theorem}\label{Theorem AnPMH}
The antiprism $A_{n}$ has the PMH-property.
\end{theorem}

\begin{proof}
Let $M$ be a perfect matching of $A_{n}$. Consider first the case when $M=\{u_{i}v_{i}: i\in[n]\}$. It is easy to see that $(v_{1},u_{1},v_{2},u_{2},\ldots,v_{n-1}, u_{n-1}, v_{n}, u_{n})$ is a Hamiltonian cycle of $A_{n}$ containing $M$. So assume that $M$ does not consist of only vertical spokes. Without loss of generality, we can assume that $M$ contains the edges $u_{n}u_{n+1}$ and $v_{n}v_{n+1}$, by Remark \ref{Remark PM in An}. We proceed by induction on $n$. The antiprism $A_{3}$ was already shown to be PMH in \cite{PMHAbreuEtAl}, since $A_{3}$ is the line graph of the complete graph $K_{4}$. So assume result is true up to $n\geq 3$, and consider $A_{n+1}$ and a perfect matching $M$ of $A_{n+1}$. Let $M'$ be $M\cup \{u_{n}v_{n}\}-\{u_{n}u_{n+1}, v_{n}v_{n+1}\}$. Then, $M'$ is a perfect matching of $A_{n}$, and so, by the inductive step, there exists a Hamiltonian cycle $H'$ of $A_{n}$ which contains $M'$. We next show that $H'$ can be extended to a Hamiltonian cycle $H$ of $A_{n+1}$ containing $M$ by considering each of the following possible induced paths in $H'$ and replacing them as indicated hereunder:
\begin{enumerate}[(i)]
\item $u_{n-1}, u_{n}, v_{n}, v_{n-1}$ is replaced by $u_{n-1}, u_{n}, u_{n+1}, v_{n+1}, v_{n}, v_{n-1}$ (similarly $u_{1}, u_{n}, v_{n}, v_{1}$ is replaced by $u_{1}, u_{n+1}, u_{n}, v_{n}, v_{n+1}, v_{1}$);
\item $u_{n-1}, v_{n}, u_{n}, u_{1}$ is replaced by $u_{n-1}, v_{n}, v_{n+1}, u_{n}, u_{n+1}, u_{1}$ (similarly $v_{n-1}, v_{n}, u_{n}, v_{1}$ is replaced by $v_{n-1}, v_{n}, v_{n+1}, u_{n}, u_{n+1}, v_{1}$); and
\item $u_{n-1}, u_{n}, v_{n}, v_{1}$ or $u_{n-1}, v_{n}, u_{n}, v_{1}$ are replaced by $u_{n-1}, v_{n}, v_{n+1}, u_{n}, u_{n+1},  v_{1}$ (similarly $v_{n-1}, v_{n}, u_{n}, u_{1}$ is replaced by $v_{n-1}, v_{n}, v_{n+1}, u_{n}, u_{n+1}, u_{1}$).
\end{enumerate}
Consequently, $A_{n+1}$ has the PMH-property, proving our theorem.
\end{proof}

\begin{theorem}
The only antiprisms having the PH-property are $A_{3}, A_{4}, A_{5}$ and $A_{6}$.
\end{theorem}

\begin{proof}
The graph $A_{3}$ is PMH as already explained in Theorem \ref{Theorem AnPMH}, so what is left to show is that every pairing $M$ of $A_{3}$ containing some edge belonging to $E(K_{A_{3}})-E(A_{3})$ (referred to as a non-edge) can be extended to a Hamiltonian cycle of $K_{A_{3}}$. The pairing $M$ can only contain one or three non-edges. When $M$ consists of three non-edges then $M=\{u_{1}v_{3}, u_{2}v_{1},u_{3}v_{2}\}$ and this can be extended to a Hamiltonian circuit of $K_{A_{3}}$ as follows: $(u_{1},v_{3},u_{3},v_{2},u_{2},v_{1})$. Otherwise, assume that the only non-edge in $M$ is $u_{1}v_{3}$, without loss of generality. Then, $M$ is either equal to $\{u_{1}v_{3}, u_{2}v_{2},u_{3}v_{1}\}$ or $\{u_{1}v_{3}, u_{2}u_{3},v_{1}v_{2}\}$, which can be extended to $(u_1,v_3,u_3,v_1,v_2,u_2)$ or $(u_1,v_3,u_3,u_2,v_2,v_1)$, respectively.
%It is an easy exercise to check that, in either case, $M$ can be extended to a Hamiltonian cycle of $K_{A_{3}}$ by using only edges of $A_3$.

The graph $A_{4}$ has the PH-property because it contains the hypercube $\mathcal{Q}_3$ as a spanning subgraph, and by the main theorem in \cite{Fink}, all hypercubes have the PH-property .

An exhaustive computer check was conducted through Wolfram Mathematica to verify that all pairings of the antiprism graphs $A_{5}$ and $A_{6}$ can be extended to a Hamiltonian cycle of the same graphs, thus proving (by brute-force) that they both have the PH-property.

\begin{figure}[h]
\centering
\includegraphics[width=0.53\textwidth]{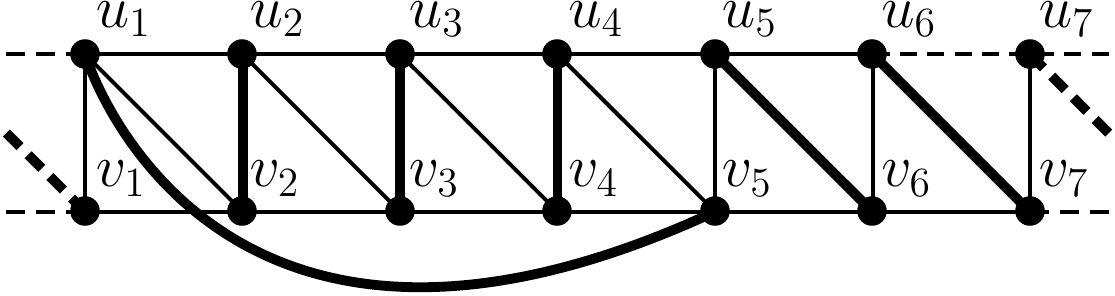}
\caption{A pairing in $A_{n}$, $n\geq 7$, which is not extendable to a Hamiltonian cycle of $K_{A_{n}}$}
\label{FigureA7notPH}
\end{figure}

The antiprism $A_{7}$ does not have the PH-property because the pairing $M=\{u_{1}v_{5}, u_{2}v_{2}, \linebreak u_{3}v_{3}, u_{4}v_{4}, u_{5}v_{6}, u_{6}v_{7}, u_{7}v_{1}\}$, depicted in Figure \ref{FigureA7notPH}, cannot be extended to a Hamiltonian cycle of $K_{A_{7}}$. For, suppose not. Let $N$ be a perfect matching of $A_{7}$ such that $M\cup N$ gives a Hamiltonian cycle of $K_{A_{7}}$. Then, $|N\cap\{u_{2}u_3, u_2v_3,v_2v_3\}|$, and $|N\cap\{u_3u_4, u_3v_4,v_3v_4\}|$ must both be equal to $1$. Consequently, $M\cup N$ induces a $M$-alternating path containing the edges $\{u_{2}v_{2},u_{3}v_{3},u_{4}v_{4}\}$ such that its endvertices are $x\in\{u_{2},v_{2}\}$ and $y\in\{u_{4},v_{4}\}$. If $xu_{1}\in N$, then $N\cap\{u_7u_1, u_1v_1,v_1v_2\}$ is empty, implying that $M\cup N$ induces the $4$-cycle $(u_6, u_7, v_1, v_7)$, a contradiction. Therefore, $N$ contains the edge $v_{2}v_{1}$ implying that $x=v_{2}$. Consequently, $N$ contains the edge $u_7u_{1}$ as well. However, this implies that $N\cap\{u_6u_7, u_7v_7,v_7v_1\}$ is empty, implying that $M\cup N$ induces the $4$-cycle $(u_5, u_6, v_7, v_6)$, a contradiction once again. Thus, $A_{7}$ does not have the PH-property.
The pairing $M$ considered above can be extended to the pairing $\{u_{1}v_{5}, u_{2}v_{2}, u_{3}v_{3}, u_{4}v_{4}\}\cup\{u_{i}v_{i+1}:i\in\{5,\ldots,n\}\}$ of $A_{n}$, for any $n>7$, and by a similar argument used for $A_7$, we conclude that $A_{n}$ does not have the PH-property for every $n\geq 7$.
\end{proof}

\subsection{$A[n,2]$}\label{subsection an2}

In \cite{ThomassenEtAl}, it is mentioned that Seongmin Ok and Thomas Perrett informed the authors that they have obtained an infinite class of $4$-regular graphs having the PH-property: such a graph in this family is obtained from a cycle of length at least three, by replacing each vertex by two isolated vertices and replacing each edge by the four edges joining the corresponding pairs of vertices. More formally, the resulting graph starting from a $n$-cycle, for $n\geq 3$, has vertex set $\{s_{i},t_{i}:i\in[n]\}$ such that, for every $i\in[n]$, $s_{i}$ and $t_{i}$ are both adjacent to $s_{i+1\pmod{n}}$ and $t_{i+1\pmod{n}}$. As far as we know, no proof of this can be found in literature, and in what follows we give a proof of this result. Before we proceed, we remark that the function that maps $s_{i}$ to $u_{i}$, and $t_{i}$ to $v_{i+1}$, for every $i\in[n]$, is an isomorphism between the above graph and the accordion graph $A[n,2]$.

\begin{theorem}\label{Theorem An2}
The accordion graph $A[n,2]$ has the PH-property, for every $n\geq 3$.
\end{theorem}

\begin{proof}
Let $A'[n,2]$ be the graph depicted in Figure \ref{FigureAAn2}, obtained from $A[n,2]$ by deleting the following set of edges: $\{u_{1}u_{n}, v_{1}v_{n}, u_{n-1}v_{1}, u_{n}v_{2}\}$. We use induction on $n$ to show that $A'[n,2]$ has the PH-property, for every $n\geq 3$. The result then follows since $A'[n,2]$ is a spanning subgraph of $A[n,2]$.

\begin{figure}[h]
\centering
\includegraphics[width=0.354\textwidth]{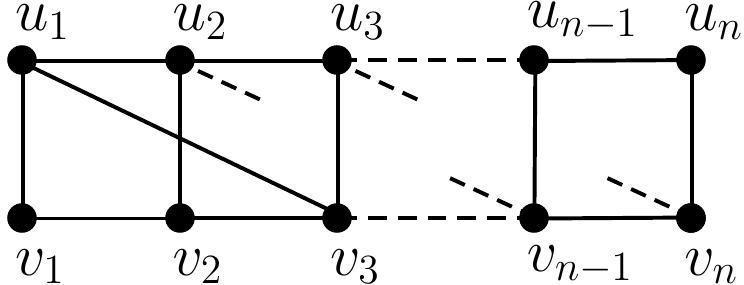}
\caption{The graph $A'[n,2]$}
\label{FigureAAn2}
\end{figure}

When $n=3$, one can show by a case-by-case analysis (or using an exhaustive computer search) that the graph $A'[3,2]$ has the PH-property. So we assume that $n>3$ and let $M$ be a pairing of $A'[n,2]$, hereafter denoted by $G$. If $M$ consists of only vertical spokes, that is, $M=\{u_{i}v_{i}:i\in[n]\}$, then
\begin{enumerate}[(i)]
\item $(v_{1},u_{1},v_{3},u_{3},\ldots, v_{n-1},u_{n-1},u_{n},v_{n},u_{n-2},v_{n-2},\ldots, u_{2},v_{2})$, when $n$ is  even, or
\item $(v_{1}, u_{1},v_{3},u_{3},\ldots, v_{n},u_{n},u_{n-1},v_{n-1},u_{n-3}\ldots,u_{2},v_{2})$, when $n$ is odd,
\end{enumerate}
is a Hamiltonian cycle of $K_G$ containing the pairing $M$. So assume that $M\neq\{u_{i}v_{i}: i\in[n]\}$. Consequently, there exists $i\in[n]$ such that $u_{i}v_{i}\not\in M$. Let $\alpha=\max\{i\in[n]: u_{i}v_{i}\not\in M\}$. We note that by a parity argument, $\alpha>1$. Consider the two subgraphs of $G$ induced by $\{u_{1}, v_{1}, \ldots, u_{\alpha-1}, v_{\alpha-1}\}$, denoted by $G_{1}$, and $\{u_{\alpha}, v_{\alpha}, \ldots, u_{n}, v_{n}\}$, denoted by $G_{2}$. We remark that $G_{1}$ and $G_{2}$ are the two components obtained by deleting from $G$ the set of edges $X$, where $X=\{u_{\alpha-1}u_{\alpha}, v_{\alpha-1}v_{\alpha},u_{\alpha-2}v_{\alpha}\}$ if $\alpha=n$, and $X=\{u_{\alpha-1}u_{\alpha}, v_{\alpha-1}v_{\alpha},u_{\alpha-2}v_{\alpha}, u_{\alpha-1}v_{\alpha+1}\}$ otherwise. We also remark that depending on the value of $\alpha$, $G_{1}$ is isomorphic to either $K_{2},C_{4}$ or $A'[\alpha-1,2]$, and $G_{2}$ is isomorphic to $K_{2}, C_{4}$ or $A'[n-(\alpha-1),2]$. Without loss of generality, we assume that $|V(G_{1})|\geq |V(G_{2})|$, implying that $3\leq\alpha\leq n$.

Next, consider the two edges $yu_{\alpha}$ and $zv_{\alpha}$ in $M$. Since for $i>\alpha$, $u_iv_i \in M$, the vertices $y$ and $z$ both belong to $\{u_{1}, v_{1}, \ldots, u_{\alpha-1}, v_{\alpha-1}\}$.
Let $M_{1}=(M\cap E(K_{G_{1}}))\cup \{yz\}$. One can see that $M_{1}$ is a pairing of $G_{1}$, and so, by the inductive step, $M_{1}$ is contained in a Hamiltonian cycle $H_{1}$ of $K_{G_{1}}$. Consequently, $H_{1}$ contains a Hamiltonian path $H_{1}'$ of $K_{G_{1}}$ with endvertices $y$ and $z$.

When $\alpha=n$, by adding the edges $yu_{\alpha}, u_{\alpha}v_{\alpha}, v_{\alpha}z$ to $E(H_{1}')$, we obtain a Hamiltonian cycle of $K_{G}$ containing $M$. For $\alpha\leq n-1$, we proceed as follows. Let $M_{2}=(M\cap E(G_{2}))\cup \{u_{\alpha}v_{\alpha}\}$. This is clearly a pairing of $G_{2}$, and so, by the inductive step, there exists a Hamiltonian cycle $H_{2}$ of $K_{G_{2}}$ containing $M_{2}$. Let $H_{2}'$ be the Hamiltonian path of $G_{2}$ obtained by deleting the edge $u_{\alpha}v_{\alpha}$ from $E(H_{2})$. Consequently, combining $H_{1}'$ and $H_{2}'$ together with the edges $yu_{\alpha}$ and $zv_{\alpha}$, we form a Hamiltonian cycle of $K_{G}$ containing $M$, as required.
\end{proof}

\section{The accordion graph $A[n,k]$ when $\gcd(n,k)\geq 5$} \label{Sect-NotPMH}

The method adopted in this section follows a similar line of thought as that used in \cite{CpCq}. Let $q=\gcd(n,k)\geq 5$, let $p=\frac{2n}{\gcd(n,k)}$ and let $p'=\frac{p}{2}$. Consider a grid-like drawing of the accordion graph $A[n,k]$ as in Remark \ref{RemarkDrawing}. For simplicity, we let the vertices $v_{1}, u_{1}, v_{1+k}, u_{1+k}, \ldots,  v_{1+(p'-1)k}, u_{1+(p'-1)k}$, be referred to as $a_{1}, a_{2}, \ldots, a_{p}$. We define the vertices $\{b_{i},c_{i},d_{i},e_{i}:i\in[p]\}$ in a similar way, where, in particular, the vertices $b_{1}, \ldots, e_{1}$, and $b_{2}, \ldots, e_{2}$, represent $v_{2}, \ldots, v_{5}$, and  $u_{2}, \ldots, u_{5}$, respectively. If $\gcd(n,k)=6$, we refer to $v_{6},u_{6},v_{6+k}, u_{6+k},\ldots, v_{6+(p'-1)k}, u_{6+(p'-1)k}$, as $f_{1},f_{2}\ldots, f_{p}$, and if $\gcd(n,k)>6$, we simply do not label all the other vertices since we are only interested in the subgraph of $A[n,k]$ induced by the vertices $\{a_i,\ldots,f_i: i\in[p]\}$. This can be seen better in Figure \ref{FigureGrid6x6}.

\begin{figure}[h]
      \centering
      \includegraphics[width=0.522\textwidth]{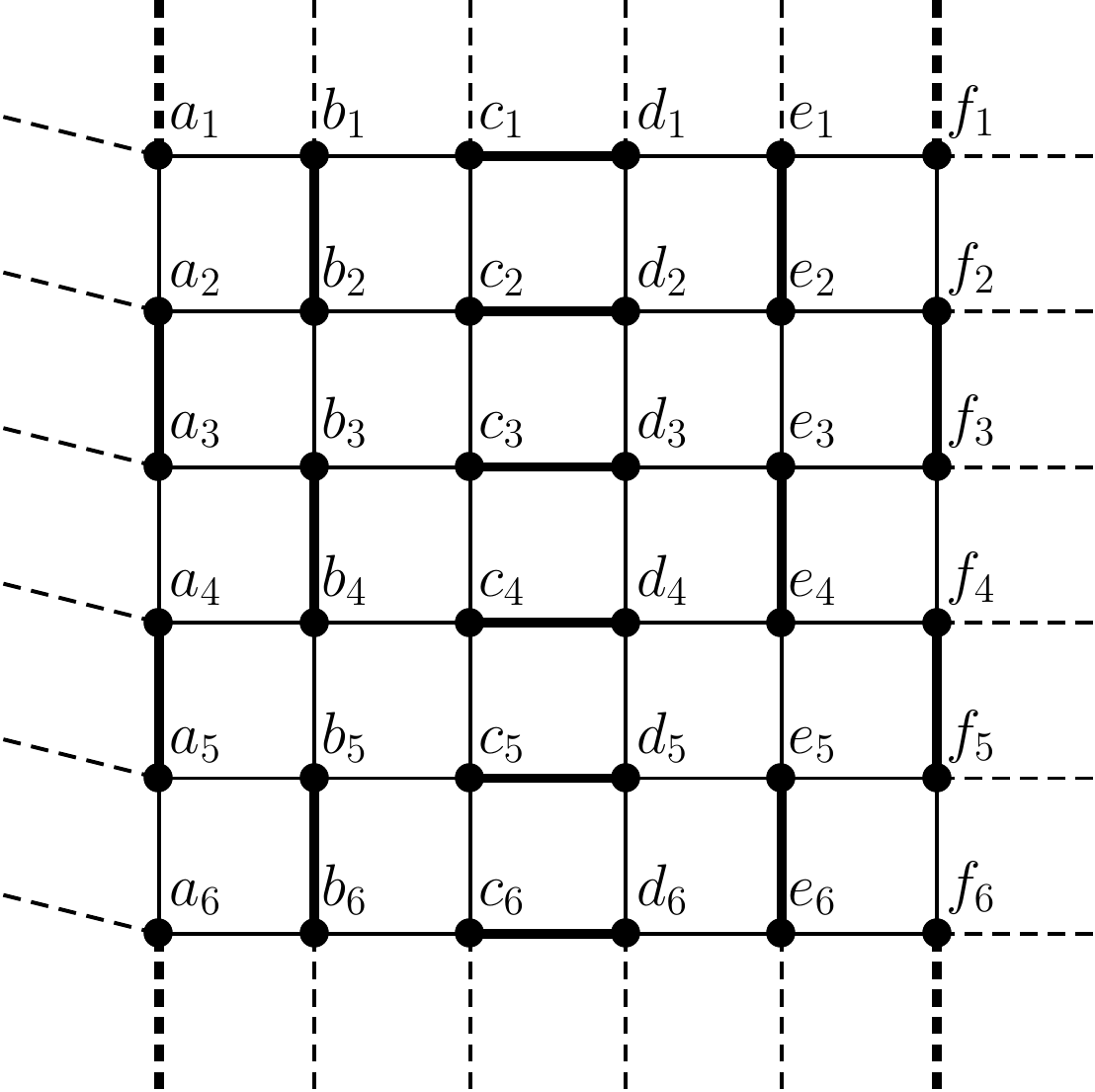}
      \caption{Edges belonging to $S$ in $A[n,k]$ when $\gcd(n,k)\geq 6$}
      \label{FigureGrid6x6}
\end{figure}

For each $i\in[p]$, let $L_{i}$ and $R_{i}$ represent the edges $b_{i}c_{i}$ and $d_{i}e_{i}$, respectively, whilst $\mathcal{L}=\{L_{i}:i\in[p]\}$ and $\mathcal{R}=\{R_{i}:i\in[p]\}$. Let  $S$ denote the following set of edges:
\begin{enumerate}[(i)]
\item $a_{i}a_{i+1}$, for every even $i\in[p]$;
\item $b_{i}b_{i+1}$ and $e_{i}e_{i+1}$, for every odd $i\in[p]$;
\item $c_{i}d_{i}$, for every $i\in[p]$; and
\item in the case when $q\geq 6$,  $f_{i}f_{i+1}$, for every even $i\in [p]$.
\end{enumerate}
Since $p$ is even, $A[n,k]$ has a perfect matching $M$ that contains $S$.

In \cite{CpCq}, it was shown that $C_p\square C_q$ is not PMH except when $p=q=4$. In the case when $q\geq 6$, the proof utilises exclusively the set $S$ of edges described above (and adapted to $C_{p}\square C_{q}$) to show that a perfect matching $M$ containing this set cannot be extended to a Hamiltonian cycle of $C_p\square C_q$. Since the same set $S$ of edges can also be chosen in a perfect matching of $A[n,k]$, the proof extends naturally and hence we have the following result.

\begin{lemma}\label{lemma gcd>=6}
The accordion graph $A[n,k]$ is not PMH if $\gcd(n,k)\geq 6$.
\end{lemma}

We shall now show that $A[n,k]$ is not PMH in the case when $\gcd(n,k)=5$. We remark that, in this case, the proof in \cite{CpCq} cannot be extended to $A[n,k]$ because it makes use of the edges $e_ia_i$ of $C_p\square C_q$, which are missing in $A[n,k]$. For the remaining part of this section, we shall need some results extracted from the proof of the main theorem in \cite{CpCq}, and adapted for accordion graphs. We note that the arguments in \cite{CpCq} are quite elaborate and lengthy, but when adapted to our case, they remain essentially the same. 
%Hence our decision not to reproduce them in detail here but to only give the main points in the following lemma.
Thus, in order not to risk obscuring the arguments that are really required for this work, we omit reproducing the same detailed discussion presented in \cite{CpCq}, and instead give the most important points emanating from the main theorem in the following self-contained lemma. 

\begin{lemma} \label{Lem-CpCq}
Let $\gcd(n,k)\geq 5$. If there exists a perfect matching $M$ of $A[n,k]$ containing $S$ and another perfect matching $N$ of $A[n,k]$ such that $M\cup N$ is a Hamiltonian cycle $H$ of $A[n,k]$, then the following statements hold.
\begin{enumerate}[(i)]
\item  $|\mathcal{L}\cap N|$ and $|\mathcal{R}\cap N|$ are both even and nonzero.
\item  A maximal sequence of consecutive edges belonging to $\mathcal{L}-N$ (or $\mathcal{R}-N$) is of even length (consecutive edges are edges having indices which are consecutive integers taken modulo $p$, with complete residue system $\{1,\ldots, p\}$).
\item  The edges of $\mathcal{L}\cap N$ are partitioned into pairs of edges $\{L_{\gamma},L_{\gamma'}\}$, where $\gamma$ is odd and $\gamma'$ is the least integer greater than $\gamma$ (taken modulo $p$) such that $L_{\gamma'}\in N$ (and similarly for $\mathcal{R}\cap N$). In this case, if we start tracing the Hamiltonian cycle $H$ from $c_{\gamma}$ going towards $b_{\gamma}$, then $H$ contains a path with edges alternating in $N$ and $M$, starting from $c_{\gamma}$ and ending at $c_{\gamma'}$, with the internal vertices on this path being $\{b_{\gamma}, b_{\gamma'}\}$, if $\gamma'=\gamma+1$, or belonging to the set $\{b_{\gamma}, a_{\gamma+1},b_{\gamma+1},\ldots, a_{\gamma'-1},b_{\gamma'-1},b_{\gamma'}\}$, if $\gamma' \neq \gamma+1$. In each of these two cases we refer to such a path between $c_{\gamma}$ and $c_{\gamma'}$ as an $L_{\gamma}L_{\gamma'}$\emph{-bracket}, or just a \emph{left-bracket}, with $L_{\gamma}$ and $L_{\gamma'}$ being the \emph{upper} and \emph{lower} edges of the bracket, respectively. \emph{Right-brackets} are defined similarly.
\item If $L_{\gamma'}$ is a lower edge of a left bracket, then $L_{\gamma'+1}$ belongs to $N$, and so is an upper edge of a possibly different left bracket (and similarly for right-brackets).
\item If $L_{i}\not\in N$, for some even $i\in [p]$, then $c_{i}c_{i+1}\in N$ (and similarly for edges belonging to $\mathcal{R}-N$).
\end{enumerate}
\end{lemma}

\begin{lemma}\label{lemma gcd=5}
The accordion graph $A[n,k]$ is not PMH if $\gcd(n,k)= 5$.
\end{lemma}

\begin{proof}

\begin{figure}[h]
\centering
\includegraphics[width=0.45\textwidth]{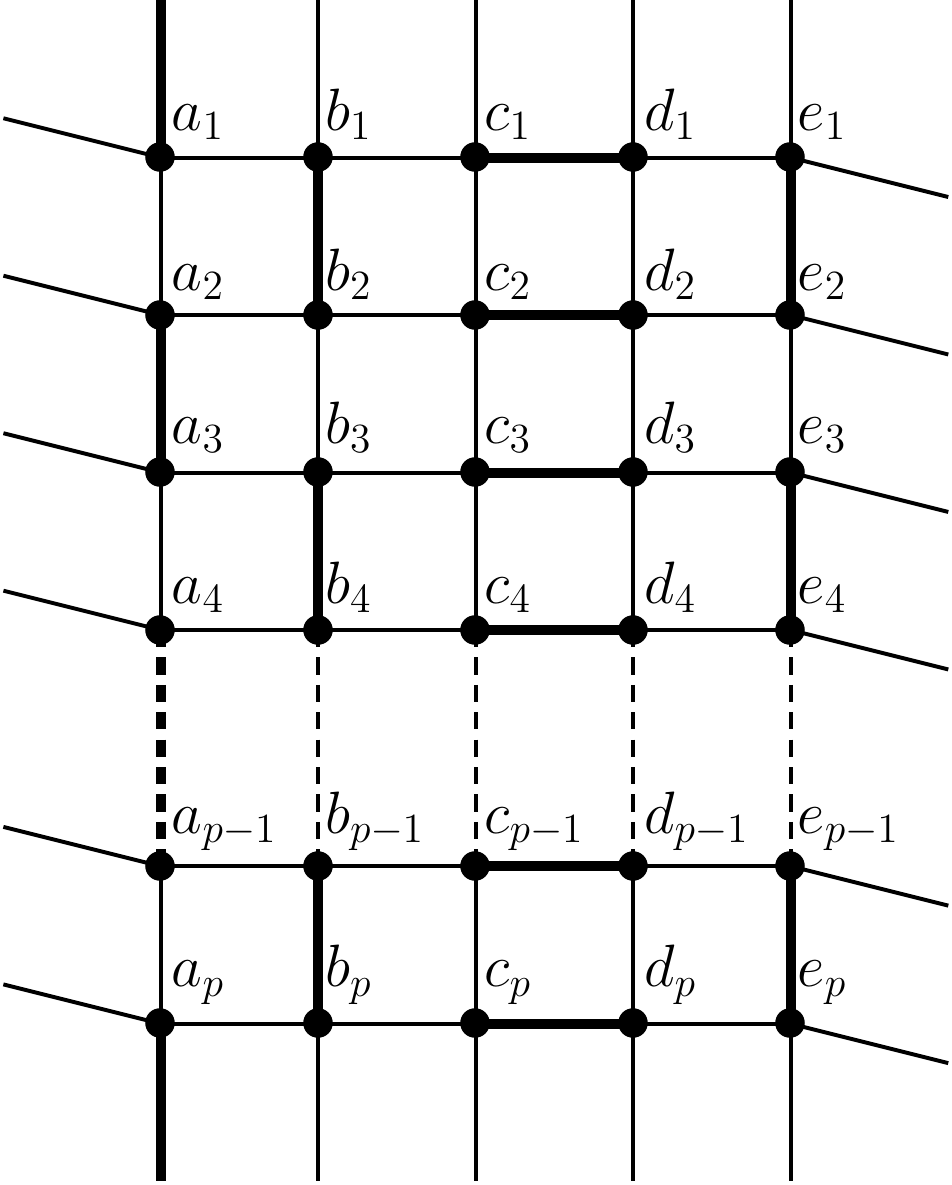}
\caption{$A[n,k]$ when $\gcd(n,k)=5$}
\label{FigureAn5}
\end{figure}

Let $M$ be a perfect matching of $A[n,k]$ which contains the set $S$ of edges as shown in Figure \ref{FigureAn5}, and let $p=\frac{2n}{5}$ (as defined earlier on in this section). Suppose that there exists a perfect matching $N$ of $A[n,k]$ such that $M\cup N$ is a Hamiltonian cycle $H$ of $A[n,k]$.

Since $\mathcal{R}\cap N\neq \emptyset$, there exists some odd $\beta \in [p]$ such that $R_{\beta}\in N$. We note that the vertex $e_{\beta}$ is adjacent to a unique $a_{\theta}$, for some odd $\theta\neq\beta$. The only way how $a_{\theta}$ can be reached by the Hamiltonian cycle $H$ is if it is reached by a left bracket, and so only if $a_{\theta}b_{\theta}\in N$. Consequently, $L_{\theta}=b_{\theta}c_{\theta}$ is not in $N$, implying that $c_{\theta}c_{\theta-1}\in N$ by Lemma \ref{Lem-CpCq}. By a similar reasoning, since $H$ is a Hamiltonian cycle, $d_{\theta}d_{\theta-1}$ cannot be contained in $N$, and so, in particular, $R_{\theta}$ and $R_{\theta-1}$ are respectively upper and lower edges belonging to $N$. Repeating the same procedure over again, first for $R_{\theta}$ and eventually for all edges in $\mathcal{R}$ having an odd index, one can deduce that $\mathcal{R}\cap N=\mathcal{R}$. Since $\mathcal{L}\cap N$ is non-empty, there exists some $\nu$ such that $L_{\nu}\in N$. The vertex $a_{\nu}$ must be reached by some right bracket in order to belong to $H$, however, this is impossible since $\mathcal{R}\cap N=\mathcal{R}$, contradicting the Hamiltonicity of $H$.
\end{proof}

By combining Lemma \ref{lemma gcd>=6} and Lemma \ref{lemma gcd=5} we obtain the main result of this section.

\begin{theorem}
The accordion graph $A[n,k]$ is not PMH if $\gcd(n,k)\geq 5$.
\end{theorem}

\section{Accordions and circulant graphs}\label{Sect-Circulant}

For integers $a$ and $b$, where $1\leq a<b\leq n-1$, $\textrm{Ci}[2n,\{a,b\}]$ denotes the \emph{quartic circulant graph} on the vertices $\{x_{i}:i\in[2n]\}$, such that $x_i$ is adjacent to the vertices in the set $\{x_{i+a},x_{i-a},x_{i+b}, x_{i-b}\}$.  We say that the edges arising from these adjacencies have length $a$ and $b$, accordingly, and we also remark that operations in the indices of the vertices $x_{i}$ are taken modulo $2n$, with complete residue system $\{1,\ldots, 2n\}$. In \cite{CpCqCirculant}, the necessary and sufficient conditions for a circulant graph to be isomorphic to a Cartesian product of two cycles were given. Motivated by this, we show that there is a non-empty intersection between the class of accordions $A[n,k]$ and the class of circulant graphs $\textrm{Ci}[2n,\{a,b\}]$, although neither one is contained in the other. In particular, in Theorem \ref{theorem accordion char circ} we show that the only accordion graphs $A[n,k]$ which are not circulant are those with both $n$ and $k$ even, such that $k\geq 4$. Results of a similar flavour about $4$-regular circulant graphs, perfect matchings and Hamiltonicity can be found in \cite{HerkeMaenhaut}.

We shall be using the following two results about circulant graphs (not necessarily quartic). Let the circulant graph on $n'$ vertices and with $r$ different edge lengths be denoted by $\textrm{Ci}[n',\{a_{1},\ldots,a_{r}\}]$. The first result, implied by a classical result in number theory, says that $\textrm{Ci}[n',\{a_{1},\ldots,a_{r}\}]$ has $\gcd(n',a_{1}, \ldots, a_{r})$ isomorphic connected components (see \cite{boeschtindell}). Consequently, in our case we have that the circulant graph $\textrm{Ci}[2n,\{a,b\}]$ is connected if and only if $\gcd(2n,a,b)=1$. Secondly, let $\gcd(n',a_{1}, \ldots, a_{r})=1$. Heuberger \cite{CirculantBipartite} showed that $\textrm{Ci}[n',\{a_{1},\ldots,a_{r}\}]$ is bipartite if and only if $a_{1}, \ldots, a_{r}$ are odd and $n'$ is even. Restated for our purposes we have the following.

\begin{corollary}\label{remark heuberger}
The quartic circulant graph $\textrm{Ci}[2n,\{a,b\}]$ is bipartite if and only if $a$ and $b$ are both odd.
\end{corollary}

Along similar lines, we show that bipartite accordion graphs can be recognised from the parity of their parameters $n$ and $k$, as follows.

\begin{lemma}\label{LemmaBipartite}
The accordion graph $A[n,k]$ is bipartite if and only if $n$ and $k$ are both even.
\end{lemma}

\begin{proof}
When $n$ and $k$ are even, the sets $\{u_{1},v_{2}, u_{3},v_{4},\ldots,u_{n-1}, v_{n}\}$ and $\{v_{1}, u_{2},\linebreak  v_{3}, u_{4},\ldots,v_{n-1},u_{n}\}$ are two independent sets of vertices of $A[n,k]$, implying that for these values of $n$ and $k$, $A[n,k]$ is bipartite. On the other hand, assume $A[n,k]$ is bipartite. Suppose that at least one of $n$ and $k$ are odd, for contradiction. If $n$ is odd it means that $(u_{1}, u_{2},\ldots, u_{n})$ is an odd cycle of $A[n,k]$, a contradiction. If $k$ is odd, then $(u_{1}, v_{1+k}, v_{k},\ldots, v_{1})$ is an odd cycle of length $k+2$, a contradiction as well.
\end{proof}

Before proceeding, we recall that operations in the indices of the vertices of $A[n,k]$ and $\textrm{Ci}[2n,\{a,b\}]$ are taken modulo $n$ and modulo $2n$, respectively.

\begin{lemma}\label{LemmaAn2}
The accordion graph $A[n,2]$ is circulant for any even integer $n\geq 4$.
\end{lemma}

\begin{proof}

For every even integer $n\geq 4$, we claim that the following function $\phi:V(A[n,2])\rightarrow V(\textrm{Ci}[2n,\{1,n-1\}])$ defined by:
\begin{itemize}
\item $\phi: u_{i}\mapsto x_{1+(n-1)(i-1)\pmod{2n}}$ for all $i\in[n]$,
\item $\phi: v_{1}\mapsto x_{2+n\pmod{2n}}$, and
\item $\phi: v_{i}\mapsto x_{2+(n-1)(i-1)\pmod{2n}}$ for $2\leq i\leq n$,
\end{itemize}
is an isomorphism. We first show that the function $\phi$ is bijective. Since $n-1\geq 3$ is odd and $\gcd(n,n-1)=1$, we have that $\gcd(2n,n-1)=1$, and so the vertices in $\{\phi(u_{i}):1\leq i\leq n\}$ are mutually distinct. For the same reason, since $2+n\equiv 1+(n-1)(2n-1)\pmod{2n}$, the vertex $\phi(v_{1})$ is distinct from any vertex $\phi(u_{i})$. Moreover, since $1\neq n-1$ and the vertices $\phi(u_{2}),\ldots,\phi(u_{n})$ are mutually distinct, the vertices in $\{\phi(v_{j}):2\leq j\leq n\}$ are mutually distinct as well, and are not equal to some vertex $\phi(u_{i})$. Finally, since $\gcd(n,n-1)=1$, we have that $2+(n-1)(i-1)\not\equiv 2+n\pmod{2n}$, for any $2\leq i\leq n$. Consequently, $\phi(v_{1})$ is distinct from any vertex $\phi(v_{j})$, proving that $\phi$ is, in fact, bijective. We next show that $\phi$ is an isomorphism. Since $\textrm{Ci}[2n,\{1,n-1\}]$ has the same number of edges as $A[n,2]$, it suffices to show that an edge in $A[n,2]$ is mapped to an edge in $\textrm{Ci}[2n,\{1,n-1\}]$.

We first take an edge $u_{i}u_{j}$ from the outer-cycle of $A[n,2]$, for some $i\in[n]$ and $j\equiv i+1\pmod{n}$, without loss of generality. Consider $\phi(u_{i})\phi(u_{j})$. The length of $\phi(u_{i})\phi(u_{j})$ can be calculated using $1+(n-1)(j-1)-(1+(n-1)(i-1))\pmod{2n}$. This is equal to $n-1$, when $j\neq 1$, and to $-n^{2}+2n-1\equiv -1\pmod{2n}$, otherwise. Since in both cases the length of the edge $\phi(u_{i})\phi(u_{j})$ belongs to $\{\pm 1, \pm(n-1)\}$, $\phi (u_{i})\phi (u_{j})\in E(\textrm{Ci}[2n,\{1,n-1\}])$.

Next, we consider the vertical spoke $u_{i}v_{i}$, for some $i\in [n]$. The length of the edge $\phi(u_{i})\phi(v_{i})$ is $2+n-(1+(n-1)(i-1))\equiv -(n-1)\pmod{2n}$, when $i=1$, and $2+(n-1)(i-1)-(1+(n-1)(i-1))=1$, otherwise, implying that the length of the edge $\phi(u_{i})\phi(v_{i})$ belongs to $\{\pm 1, \pm(n-1)\}$ in both cases. Consequently, $\phi (u_{i})\phi (v_{i})\in E(\textrm{Ci}[2n,\{1,n-1\}])$.

We now take the diagonal spoke $u_{i}v_{j}$, for some $i\in[n]$, and for $j\equiv i+2\pmod{n}$. The length of the edge $\phi(u_{i})\phi(v_{j})$ is:
\begin{itemize}
\item $2+(n-1)(j-1)-(1+(n-1)(i-1))= 2n-1\equiv -1 \pmod{2n}$, when $1\leq i\leq n-2$;
\item $2+n-(1+(n-1)(i-1))= -n^{2}+4n-1 \equiv -1\pmod{2n}$, when $i=n-1$; and
\item $2+(n-1)(j-1)-(1+(n-1)(i-1))= -n^{2}+3n-1 \equiv n-1\pmod{2n}$, when $i=n$.
\end{itemize}
In each of these three cases, the length of $\phi(u_{i})\phi(v_{j})$ belongs to $\{\pm 1, \pm(n-1)\}$, implying that $\phi (u_{i})\phi (v_{j})\in E(\textrm{Ci}[2n,\{1,n-1\}])$.

Finally, we take an edge $v_{i}v_{j}$ from the inner-cycle of $A[n,2]$, for some $i\in[n]$ and $j\equiv i+1\pmod{n}$, without loss of generality. The length of the edge $\phi(v_{i})\phi(v_{j})$ is:

\begin{itemize}
\item $2+n-(2+(n-1)(i-1))= -n^{2}+3n-1\equiv n-1\pmod{2n}$, when $j=1$;
\item $2+(n-1)(j-1)-(2+n)=-1$, when $j=2$; and
\item $2+(n-1)(j-1)-(2+(n-1)(i-1))=n-1$, when $3\leq j \leq n$.
\end{itemize}
This implies that the length of the edge $\phi(v_{i})\phi(v_{j})$ belongs to $\{\pm 1, \pm(n-1)\}$ in both cases, and so $\phi (v_{i})\phi (v_{j})\in E(\textrm{Ci}[2n,\{1,n-1\}])$.\\

There are no more cases to consider, proving our result.
\end{proof}

We next show that the only circulant accordions $A[n,k]$ with both $n$ and $k$ even are the ones having $k=2$. By Lemma \ref{LemmaBipartite}, this means that the only circulant bipartite accordions are the ones with $n$ even and $k=2$.

\begin{lemma}\label{Lemma n,k even}
For $n$ and $k$ even, the accordion graph $A[n,k]$ is circulant if and only if $k=2$.
\end{lemma}

\begin{proof}
Let $n$ and $k$ be even. By Lemma \ref{LemmaAn2}, it suffices to show that accordion graphs admitting $k\geq 4$ are not circulant. We can further assume that $n\geq 8$, since the only accordions with $n=4$ or $6$, and $k$ even are $A[4,2]$ and $A[6,2]$. Suppose, for contradiction, that for $n\geq 8$, there exists an even integer $k\geq 4$ such that $A[n,k]$ is circulant. Then, by Corollary \ref{remark heuberger} and Lemma \ref{LemmaBipartite}, $A[n,k]\simeq  \textrm{Ci}[2n,\{a,b\}]$ for some distinct odd integers $a$ and $b$. For simplicity, we refer to $A[n,k]$, or equivalently $\textrm{Ci}[2n,\{a,b\}]$, by $G$. By the definition of quartic circulant graphs, we recall that $1\leq a<b\leq n-1$. \\

\noindent\textbf{Claim.} $\gcd(2n,a)=\gcd(2n,b)=1$.

\noindent\emph{Proof of Claim.} Suppose that $\gcd(2n,a)\not=1$, for contradiction. Then, the least common multiple of $a$ and $2n$ is $2na'$, for some $a'<a$. Consequently, there exists an even integer $p$ such that $ap=2na'$. Moreover, since $a\neq a'$ and $a$ is odd, $\frac{a}{a'}$ (or equivalently $\frac{2n}{p}$) is odd and at least $3$, and so $p<n$.
By considering the edges in $G$ having length $a$, there exists a partition $\mathcal{P}$ of the $2n$ vertices of $G$ into $\frac{2n}{p}$ sets, each inducing a $p$-cycle. This follows since $\gcd(2n,a)=\frac{2n}{p}\not=1$. Furthermore, $\mathcal{P}$ has an odd number of components, namely $\gcd(2n,a)$, or equivalently $\frac{2n}{p}$.

Since $G$ is a connected quartic graph and $\frac{2n}{p}>1$, two vertices on a particular $p$-cycle in $\mathcal{P}$ are adjacent in $G$ if and only if there is an edge of length $a$ between them, in other words, the subgraph induced by the vertices on a $p$-cycle in $\mathcal{P}$ is the $p$-cycle itself.
Therefore, the graph contains two adjacent vertices $x_{i}$ and $x_{j}$ belonging to two different $p$-cycles from $\mathcal{P}$. Consequently, since $i\equiv j\pm b\pmod{2n}$, the vertices of these two $p$-cycles induce $C_{p}\square P_{2}$. By a similar argument to that used on $x_{i}$ and $x_{j}$, we deduce that $G$ contains a spanning subgraph $G_{0}$ isomorphic to $C_{p}\square P_{\frac{2n}{p}}$.

We now denote the set $\{u_{1},u_{2},\ldots, u_{n}\}$ of vertices on the outer-cycle of $A[n,k]$ by $\mathcal{U}$, and the set $\{v_{1},v_{2},\ldots, v_{n}\}$ of vertices on the inner-cycle of $A[n,k]$ by $\mathcal{V}$, and claim that:
\begin{enumerate}[(i)]
\item given two adjacent vertices from some $p$-cycle in $\mathcal{P}$, say $x_{i}$ and $x_{i+a}$, if $x_{i}$ is a vertex in $\mathcal{U}$, then $x_{i+a}$ is a vertex in $\mathcal{V}$, or vice-versa; and
\item given two adjacent vertices from two different $p$-cycles in $\mathcal{P}$, say $x_{i}$ and $x_{i+b}$, we have that either both belong to $\mathcal{U}$ or both belong to $\mathcal{V}$.
\end{enumerate}
First of all, we note that the vertices inducing a $p$-cycle from $\mathcal{P}$, cannot all belong to $\mathcal{U}$, since the latter set of vertices induces a $n$-cycle, and $p<n$. Similarly, the vertices inducing a $p$-cycle from $\mathcal{P}$, cannot all belong to $\mathcal{V}$. Secondly, let $i\in[2n]$, such that $x_{i}$ is of degree $3$ in $G_{0}$, and $x_{i}x_{i+b}\in E(G_{0})$. Consider the $4$-cycle $(x_{i}, x_{i+a},x_{i+a+b}, x_{i+b})$. Since $n>4$, these four vertices cannot all belong to $\mathcal{U}$ (or $\mathcal{V}$). Also, we cannot have three of them which belong to $\mathcal{U}$ (or $\mathcal{V}$), because otherwise we would have $k=2$, or, $k\equiv -2\pmod n$, that is, $k=n-2$. Since we are assuming that $k\geq 4$, we must have $k=n-2$, but by Definition \ref{def AccordionGraphs}, $k$ is at most $\frac{n}{2}$, and so, $n-2\leq\frac{n}{2}$, a contradiction, since $n\geq 8$.
This means that exactly two vertices from $(x_{i}, x_{i+a},x_{i+a+b}, x_{i+b})$ belong to $\mathcal{U}$, and the other two belong to $\mathcal{V}$. Without loss of generality, assume that $x_{i}$ belongs to $\mathcal{U}$.

Suppose that $x_{i+b}\not\in\mathcal{U}$, for contradiction. Then, $\mathcal{U}$ must contain exactly one of $x_{i+a}$ and $x_{i+a+b}$. Suppose we have $x_{i+a+b}\in\mathcal{U}$. Consequently, $x_{i+a}$ and $x_{i+b}$ belong to $\mathcal{V}$, and so since $\frac{2n}{p}\geq 3$, $x_{i+2a+b}$ and $x_{i+a+2b}$ belong to $\mathcal{U}$, giving rise to a $4$-cycle in $G_{0}$ with three of its vertices belonging to $\mathcal{U}$, a contradiction.
Therefore, we have $x_{i+a}\in\mathcal{U}$. This means that $x_{i+b}$ and $x_{i+a+b}$ both belong to $\mathcal{V}$. Suppose further that $x_{i+2a}\in\mathcal{V}$. Since we cannot have three vertices in a $4$-cycle belonging to $\mathcal{V}$, $x_{i+2a+b}\in \mathcal{U}$. However, this once again gives rise to a $4$-cycle in $G_{0}$ with three of its vertices belonging to $\mathcal{U}$, a contradiction. Therefore, $x_{i+2a}$ must belong to $\mathcal{U}$. By repeating the same argument used for $x_{i+2a}$ to the vertices $x_{i+3a},\ldots, x_{i+(p-1)a}$, we get that all the vertices in the $p$-cycle, from $\mathcal{P}$, containing $x_{i}$ belong to $\mathcal{U}$, a contradiction. Hence, $x_{i+a}\not\in\mathcal{U}$. Thus, neither one of $x_{i+a}$ and $x_{i+a+b}$ is in $\mathcal{U}$, contradicting our initial assumption. This implies that $x_{i+b}\in\mathcal{U}$, and that $x_{i\pm a}$ and $x_{i+b\pm a}$ belong to $\mathcal{V}$. This forces all the vertices not considered so far to satisfy the two conditions in the above claim.

Thus, by Remark \ref{RemarkDrawing}, $\frac{2n}{p}=\gcd(n,k)$. This is a contradiction, since $\frac{2n}{p}$ is odd and the greatest common divisor of two even numbers is even. Hence, $\gcd(2n,a)=1$, and by a similar reasoning, $\gcd(2n,b)=1$ as well. \,\,\,{\tiny$\blacksquare$}\\

This implies that $a$ does not divide $n$, $b$ does not divide $n$, and that the edges of $G$ can be partitioned in two Hamiltonian cycles induced by the edges having length $a$ and $b$, respectively.

In particular, since $a$ does not divide $2n$, there exists an edge of length $a$ with both endvertices belonging to $\{u_{i}:i\in[n]\}\subset V(A[n,k])$. Without loss of generality, assume that $u_{1}u_{2}$ has length $a$, and consider the $4$-cycle $C=(u_{1},u_{2},v_{2},v_{1})$. Since the edges having length $a$ (and similarly the edges having length $b$) induce a Hamiltonian cycle, and $n>4$, the lengths of the edges in $C$ cannot all be the same. Hence, the lengths of the edges $(u_{1}u_{2}, u_{2}v_{2}, v_{2}v_{1}, v_{1}u_{1})$ of $C$ can be of
Type A1 $:=(a,b,b,b)$, Type A2 $:=(a,a,b,a)$, Type A3 $:=(a,a,a,b)$, Type A4 $:=(a,b,a,a)$, Type  B1 $:=(a,a,b,b)$, Type B2 $:=(a,b,b,a)$, or Type B3 $:=(a,b,a,b)$. Some of these types are depicted in Figure \ref{FigureTypes}.

\begin{figure}[h]
\centering
\includegraphics[width=.82\textwidth, keepaspectratio]{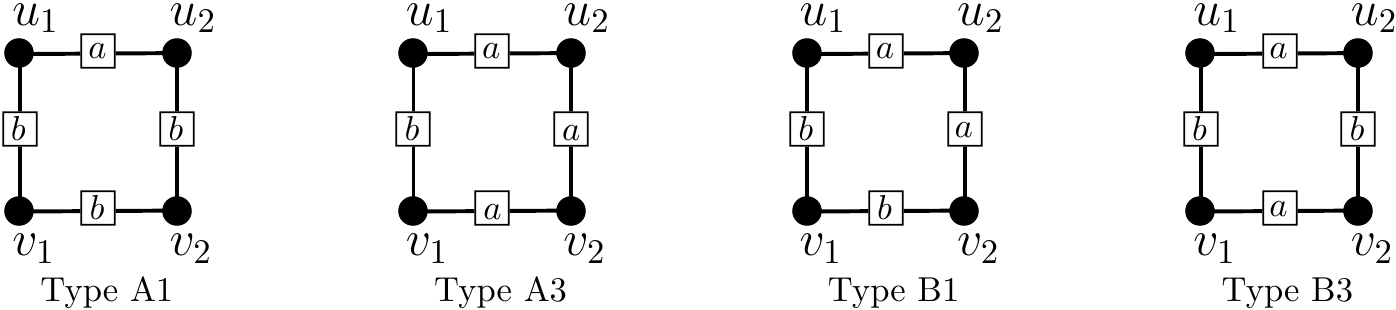}
\caption[Different lengths of the edges in $C$ from the proof of Lemma \ref{Lemma n,k even}]{Different lengths of the edges in $C$}
\label{FigureTypes}
\end{figure}

If $C$ is of Type A1, then $a\equiv\pm 3b\pmod{2n}$. This implies that the two endvertices of an edge of length $a$ are also endvertices of a $3$-path whose edges are all of length $b$. Also, the two endvertices of a $3$-path whose edges are all of length $b$ must be adjacent. Consider the edge $v_{2}v_{3}$. Since $v_{1}v_{2}$ and $u_{2}v_{2}$ have length $b$, the edge $v_{2}v_{3}$ has length $a$, and so it must belong to some $4$-cycle with the other three edges of the cycle having length $b$. We denote this $4$-cycle by $C_{4}(v_{2}v_{3})$. First, assume that $u_{2}u_{3}$ is of length $a$. If $u_{2}v_{2}\in E(C_{4}(v_{2}v_{3}))$, then, $E(C_{4}(v_{2}v_{3}))$ contains $u_{2}v_{2+k}$ and consequently $v_{2+k}v_{3}$, which is impossible, since $k\geq 4$. Therefore, $E(C_{4}(v_{2}v_{3}))$ contains $v_{1}v_{2}$, and so $C_{4}(v_{2}v_{3})=(v_{3},v_{2},v_{1},u_{1})$, implying that $k=2$, a contradiction. Consequently, $u_{2}u_{3}$ must be of length $b$, implying that $C_{4}(v_{2}v_{3})=(v_{3},v_{2},u_{2},u_{3})$. By using the same arguments we can deduce that the outer- and inner-cycle edges, and the vertical spokes in $G$ have lengths as shown in Figure \ref{FigureTypeA1}.

\begin{figure}[h]
\centering
\includegraphics[width=0.55\textwidth, keepaspectratio]{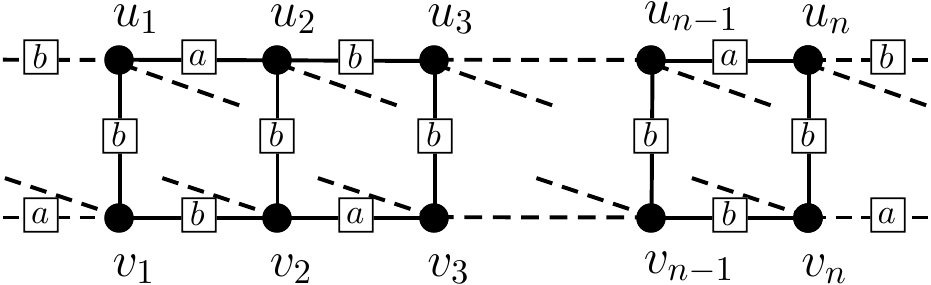}
\caption{$A[n,k]$ when $C$ is of Type A1}
\label{FigureTypeA1}
\end{figure}

Since $k$ is even, we also have $b\equiv\pm 3a\pmod{2n}$ (see for example the $4$-cycle $(u_{1},u_{2},\linebreak  v_{2+k},v_{1+k})$). This implies that $a\equiv\pm 9a\pmod{2n}$, that is, $8a\equiv 2n\pmod{2n}$, or $10a\equiv 2n\pmod{2n}$. Since $\gcd(2n,a)=1$, the total number of vertices of $G$ must be equal to $8$ or $10$, a contradiction, since $n\geq 8$.

If $C$ is of Type A2, then $b\equiv\pm 3a\pmod{2n}$. This implies that the two endvertices of an edge of length $b$ are also endvertices of a $3$-path whose edges are all of length $a$. Also, the two endvertices of a $3$-path whose edges are all of length $a$ must be adjacent. Consider the edge $u_{2}u_{3}$. Since $u_{1}u_{2}$ and $u_{2}v_{2}$ have length $a$, the edge $u_{2}u_{3}$ has length $b$, and so it must belong to some $4$-cycle with the other three edges of the cycle having length $a$. We denote this $4$-cycle by $C_{4}(u_{2}u_{3})$. First, assume that $v_{2}v_{3}$ is of length $b$. If $u_{2}v_{2}\in E(C_{4}(u_{2}u_{3}))$, then, $E(C_{4}(u_{2}u_{3}))$ contains $v_{2}u_{2-k}$ and consequently $u_{2-k}u_{3}$. This means that $u_{2-k}=u_{4}$, and so, since $u_{4}$ is adjacent to $v_{2}$ we have that $k\equiv-2\pmod{n}$. By Definition \ref{def AccordionGraphs}, this implies that $k=n-2\leq\frac{n}{2}$, a contradiction since $n\geq 8$. Therefore, $E(C_{4}(u_{2}u_{3}))$ contains $u_{1}u_{2}$, and so $C_{4}(u_{2}u_{3})=(u_{3},u_{2},u_{1},v_{1})$, implying once again that $k\equiv-2\pmod{n}$, a contradiction. Consequently, $v_{2}v_{3}$ must be of length $a$, implying that $C_{4}(u_{2}u_{3})=(u_{3},u_{2},v_{2},v_{3})$. By using the same arguments we can deduce that the outer- and inner-cycle edges, and the vertical spokes in $G$ have lengths as shown in Figure \ref{FigureTypeA2}.

\begin{figure}[h]
\centering
\includegraphics[width=0.55\textwidth, keepaspectratio]{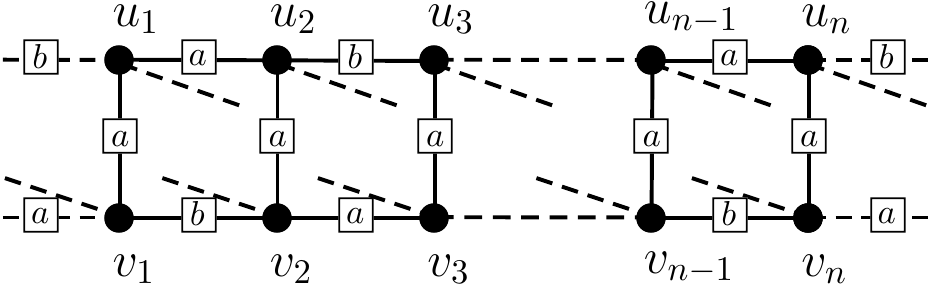}
\caption{$A[n,k]$ when $C$ is of Type A2}
\label{FigureTypeA2}
\end{figure}

Since $k$ is even, we also have $a\equiv\pm 3b\pmod{2n}$ (see for example the $4$-cycle $(u_{2},u_{3},\linebreak  v_{3+k},v_{2+k})$). This once again implies that $a\equiv\pm 9a\pmod{2n}$, a contradiction as in the case when $C$ is of Type A1.

So assume that $C$ is of Type A3. Then, $b\equiv\pm 3a\pmod{2n}$ and, in particular, the edge $u_{2}u_{3}$ has length $b$. Consequently, this edge must belong to some $4$-cycle with the other three edges of the cycle having length $a$. We denote this $4$-cycle by $C_{4}(u_{2}u_{3})$. Since $u_{1}u_{2}$ and $u_{2}v_{2}$ are both of length $a$, we have the following cases:
\begin{itemize}
\item if $C_{4}(u_{2}u_{3})=(u_{3},u_{2},u_{1},v_{1+k})$, then $k=2$, a contradiction;
\item if $C_{4}(u_{2}u_{3})=(u_{3},u_{2},u_{1},u_{n})$, then $n=4$, a contradiction; and
\item if $C_{4}(u_{2}u_{3})=(u_{3},u_{2},v_{2},v_{1})$, then $u_{3}$ is adjacent to $v_{1}$. Consequently, we have that $k\equiv -2\pmod n$, and as before, this implies that $n-2\leq\frac{n}{2}$, a contradiction, since $n\geq 8$.
\end{itemize}
Thus $C$ cannot be of Type A3, and by using a similar argument, it can be shown that $C$ cannot be of Type A4 either.
In fact, if $C$ is of Type A4, then, $b\equiv\pm 3a\pmod{2n}$ once again, and, in particular, the edge $u_{1}u_{n}$ has length $b$. Consequently, this edge must belong to some $4$-cycle with the other three edges of the cycle having length $a$. We denote this $4$-cycle by $C_{4}(u_{1}u_{n})$. Since $u_{1}u_{2}$ and $u_{1}v_{1}$ are both of length $a$, we have the following cases:
\begin{itemize}
\item if $C_{4}(u_{1}u_{n})=(u_{n},u_{1},v_{1},v_{2})$, then $k=2$, a contradiction;
\item if $C_{4}(u_{1}u_{n})=(u_{n},u_{1},u_{2},u_{3})$, then $n=4$, a contradiction; and
\item if $C_{4}(u_{1}u_{n})=(u_{n},u_{1},u_{2},v_{2+k})$, then $u_{n}$ is adjacent to $v_{2+k}$. Consequently, we have that $k\equiv -2\pmod n$, and as before, this implies that $n-2\leq\frac{n}{2}$, a contradiction, since $n\geq 8$.
\end{itemize}
Thus, $C$ cannot be of Type A4.
If $C$ is of Type B1 or Type B2, then we have that $2a\equiv\pm 2b\pmod{2n}$, and since $1\leq a<b\leq n-1$, we can further assume that $2a\equiv -2b\pmod{2n}$. Consequently, we have that $a+b=n$, and so, since $\gcd(2n,a)=\gcd(2n,b)=\gcd(2n,1)=\gcd(2n,n-1)=1$, by Lemma \ref{LemmaAn2}, $G\simeq A[n,2]$, a contradiction. Therefore, $C$ and  all other possible $4$-cycles in $G$ must be of Type B3, which is impossible, because then the edges having length $a$ would induce two disjoint $n$-cycles, contradicting the fact that the edges having length $a$ induce a Hamiltonian cycle (and thus a $2n$-cycle). As a consequence, $A[n,k]$ is not circulant, contradicting our initial assumption.
\end{proof}

Using the above two lemmas we can now prove the main result of this section.

\begin{theorem}\label{theorem accordion char circ}
The accordion graph $A[n,k]$ is not circulant if and only if both $n$ and $k$ are even, such that $k\geq 4$.
\end{theorem}

\begin{proof} By Lemma \ref{Lemma n,k even}, it suffices to show that the accordion graph $A[n,k]$ is circulant if and only if either
\begin{enumerate}[(i)]
\item $k$ is odd, or
\item $k$ is even and $n$ is odd, or
\item $k=2$ and $n$ is even.
\end{enumerate}

\noindent\textbf{Case (i).}  For $k$ odd, we claim that the function $\phi:V(A[n,k])\rightarrow   V(\textrm{Ci}[2n,\{2,k\}])$ defined by $\phi: u_{i}\mapsto x_{2i}$ and $\phi: v_{i}\mapsto x_{2i-k\pmod{2n}}$, where $i\in [n]$, is an isomorphism. Since $2i-k$ is odd, for every $i\in[n]$, one can deduce that the function $\phi$ is bijective. Also, $\textrm{Ci}[2n,\{2,k\}]$ has the same number of edges as $A[n,k]$, and thus it suffices to show that an edge in $A[n,k]$ is mapped to an edge in $\textrm{Ci}[2n,\{2,k\}]$.

\begin{enumerate}[(a)]
\item We first take an edge $u_{i}u_{j}$ from the outer-cycle of $A[n,k]$, for some $i\in[n]$ and $j\equiv i+1\pmod{n}$, without loss of generality. Consider $\phi(u_{i})\phi(u_{j})$. The length of $\phi(u_{i})\phi(u_{j})$ can be calculated using $2(j-i)$ which is equivalent to $2\pmod{2n}$. Since this belongs to $\{\pm 2, \pm k\}$, $\phi (u_{i})\phi (u_{j})\in E(\textrm{Ci}[2n,\{2,k\}])$.

\item By a similar reasoning to that used in (a), $\phi(v_{i})\phi(v_{j})$ is an edge in $\textrm{Ci}[2n,\{2,k\}]$, for any $i\in[n]$ and $j\equiv i+1\pmod{n}$, without loss of generality.

\item We now consider the spokes. Let $i\in[n]$ and $j\equiv i+k\pmod{n}$. The length of $\phi(u_{i})\phi(v_{i})$ can be calculated using $2i-k-2i$, which is equal to $-k$. On the other hand, the length of $\phi(u_{i})\phi(v_{j})$ can be calculated using $2j-k-2i$, which is equal to $k$. In both cases, the lengths obtained belong to $\{\pm 2, \pm k\}$, and so $\phi(u_{i})\phi(v_{i})$ and $\phi(u_{i})\phi(v_{j})$ are edges in $\textrm{Ci}[2n,\{2,k\}]$.

\end{enumerate}

\noindent\textbf{Case (ii).} For $k$ even and $n$ odd, we claim that the following function $\phi:V(A[n,k])\rightarrow V(\textrm{Ci}[2n,\{2,n-k\}])$ defined by $\phi: u_{i}\mapsto x_{2i}$ and $\phi: v_{i}\mapsto x_{2i+n-k\pmod{2n}}$, where $i\in [n]$, is an isomorphism. As in Case (i), the function $\phi$ is bijective, since $n-k$ is odd. Moreover, $\textrm{Ci}[2n,\{2,n-k\}]$ has the same number of edges as $A[n,k]$, and so it suffices to show that an edge in $A[n,k]$ is mapped to an edge in $\textrm{Ci}[2n,\{2,n-k\}]$.

\begin{enumerate}[(a)]

\item By the same reasoning used in Case (i), $\phi(u_{i})\phi(u_{j})$ and $\phi(v_{i})\phi(v_{j})$ are edges in $\textrm{Ci}[2n,\{2,n-k\}]$, for $i\in[n]$, and $j\equiv i+1\pmod{n}$, without loss of generality.

\item We now consider the spokes. Let $i\in[n]$ and $j\equiv i+k\pmod{n}$. The length of $\phi(u_{i})\phi(v_{i})$ can be calculated using $2i+n-k-2i$, which is equal to $n-k$. On the other hand, the length of $\phi(u_{i})\phi(v_{j})$ can be calculated using $2j+n-k-2i$, which is equivalent to $-(n-k)\pmod{2n}$. In both cases, the lengths obtained belong to $\{\pm 2, \pm (n-k)\}$, and so $\phi(u_{i})\phi(v_{i})$ and $\phi(u_{i})\phi(v_{j})$ are edges in $\textrm{Ci}[2n,\{2,n-k\}]$.
\end{enumerate}
\noindent\textbf{Case (iii).} This was proven in Lemma \ref{LemmaAn2}.
\end{proof}

The following result follows immediately from the proof of Theorem \ref{theorem accordion char circ}.

\begin{corollary}
The accordion graph $A[n,k]$ is isomorphic to the circulant graph
\begin{enumerate}[(i)]
\item $\textrm{Ci}[2n,\{2,k\}]$, when $k$ is odd;
\item $\textrm{Ci}[2n,\{2,n-k\}]$, when $n$ is odd and $k$ is even; and
\item $\textrm{Ci}[2n,\{1,n-1\}]$, when $n$ is even and $k=2$.
\end{enumerate}
\end{corollary}

\section{Concluding remarks and open problems} \label{Sect-Concl}

Despite ruling out all accordion graphs $A[n,k]$ having $\gcd(n,k)\geq 5$, a complete characterisation of which accordion graphs have the PMH- or the PH-property is definitely of interest but still inaccessible. In Section \ref{Sect-PH}, partial results were obtained for the cases when $\gcd(n,k)\leq 2$. These are portrayed in Table \ref{table acc pmh} together with other partial results obtained by a computer check conducted through Wolfram Mathematica. In particular, we identify which accordions are PMH and which are not, for $1\leq k\leq 10$ and for $n\leq 21$. We remark that some values of $n$ and $k$ are marked as ``unknown" due to problems with computation time and memory.

\begin{table}[h]
\centering
\begin{tabular}{ccccccccccccccl}
\addlinespace[-\aboverulesep]
\cmidrule[\heavyrulewidth]{1-12}%\toprule
\multicolumn{2}{l}{$A[n,k]$}      & \multicolumn{9}{c}{$k$}                   \\
\cmidrule{1-12}
      &       & $1$    & $2$   & $3$    & $4$    & $5$    & $6$    & $7$    & $8$    & $9$&$10$&&&\\
\multirow{19}[0]{*}{$n$}
&$3$		  & \cellcolor[gray]{0.85}  & \cellcolor[gray]{0}  & \cellcolor[gray]{0}  & \cellcolor[gray]{0}  & \cellcolor[gray]{0}  & \cellcolor[gray]{0}  & \cellcolor[gray]{0}  & \cellcolor[gray]{0}  & \cellcolor[gray]{0}& \cellcolor[gray]{0} & & & \\
&$4$		  & \cellcolor[gray]{0.85}  & \cellcolor[gray]{0.85}  & \cellcolor[gray]{0}  & \cellcolor[gray]{0}  & \cellcolor[gray]{0}  & \cellcolor[gray]{0}  & \cellcolor[gray]{0}  & \cellcolor[gray]{0}  & \cellcolor[gray]{0}& \cellcolor[gray]{0} &&& \\
&$5$		  & \cellcolor[gray]{0.85}  & \cellcolor[gray]{0.85}  & \cellcolor[gray]{0}  & \cellcolor[gray]{0}  & \cellcolor[gray]{0}  & \cellcolor[gray]{0}  & \cellcolor[gray]{0}  & \cellcolor[gray]{0}  & \cellcolor[gray]{0}& \cellcolor[gray]{0} &&& \\
&$6$		  & \cellcolor[gray]{0.85}  & \cellcolor[gray]{0.85}  & \cellcolor[gray]{0.85}  & \cellcolor[gray]{0}  & \cellcolor[gray]{0}  & \cellcolor[gray]{0}  & \cellcolor[gray]{0}  & \cellcolor[gray]{0}  & \cellcolor[gray]{0}& \cellcolor[gray]{0} &&& \\
&$7$          & \cellcolor[gray]{0.85}  & \cellcolor[gray]{0.85}  & \cellcolor[gray]{0.85}  & \cellcolor[gray]{0}  & \cellcolor[gray]{0}  & \cellcolor[gray]{0}  & \cellcolor[gray]{0}  & \cellcolor[gray]{0}  & \cellcolor[gray]{0}& \cellcolor[gray]{0} &&& \\
&$8$          & \cellcolor[gray]{0.85}  & \cellcolor[gray]{0.85}  & \cellcolor[gray]{0.85}  & \cellcolor[gray]{0.85}  & \cellcolor[gray]{0}  & \cellcolor[gray]{0}  & \cellcolor[gray]{0}  & \cellcolor[gray]{0}  & \cellcolor[gray]{0}& \cellcolor[gray]{0} &&& \\
&$9$          & \cellcolor[gray]{0.85}  & \cellcolor[gray]{0.85}  & \cellcolor[gray]{0.85}  & \cellcolor[gray]{0.85}  & \cellcolor[gray]{0}  & \cellcolor[gray]{0}  & \cellcolor[gray]{0}  & \cellcolor[gray]{0}  & \cellcolor[gray]{0}& \cellcolor[gray]{0} &&& \\
&$10$         & \cellcolor[gray]{0.85}  & \cellcolor[gray]{0.85}  & \cellcolor[gray]{0.85}  & \cellcolor[gray]{0.85}  & $\bot$ & \cellcolor[gray]{0}  & \cellcolor[gray]{0}  & \cellcolor[gray]{0}  & \cellcolor[gray]{0}& \cellcolor[gray]{0} &&& \\
&$11$         & \cellcolor[gray]{0.85}  & \cellcolor[gray]{0.85}  & \cellcolor[gray]{0.85}  & \cellcolor[gray]{0.85}  & \cellcolor[gray]{0.85}  & \cellcolor[gray]{0}  & \cellcolor[gray]{0}  & \cellcolor[gray]{0}  & \cellcolor[gray]{0}& \cellcolor[gray]{0} &&\cellcolor[gray]{0.85} & $A[n,k]$ PMH \\
&$12$         & \cellcolor[gray]{0.85}  & \cellcolor[gray]{0.85}  & \cellcolor[gray]{0.85}  & \cellcolor[gray]{0.85}  & \cellcolor[gray]{0.85}  & $\bot$ & \cellcolor[gray]{0}  & \cellcolor[gray]{0}  & \cellcolor[gray]{0}& \cellcolor[gray]{0} &&$\bot$ & $A[n,k]$ not PMH \\
&$13$         & \cellcolor[gray]{0.85}  & \cellcolor[gray]{0.85}  & $\bot$ & \cellcolor[gray]{0.85}  & \cellcolor[gray]{0.85}  & \cellcolor[gray]{0.85}  & \cellcolor[gray]{0}  & \cellcolor[gray]{0}  & \cellcolor[gray]{0}& \cellcolor[gray]{0} &&?&Unknown \\
&$14$         & \cellcolor[gray]{0.85}  & \cellcolor[gray]{0.85}  & $\bot$ & \cellcolor[gray]{0.85}  & \cellcolor[gray]{0.85}  & \cellcolor[gray]{0.85}  & $\bot$ & \cellcolor[gray]{0}  & \cellcolor[gray]{0}& \cellcolor[gray]{0} &&& \\
&$15$         & \cellcolor[gray]{0.85}  & \cellcolor[gray]{0.85}  & $\bot$ & $\bot$ & $\bot$ & \cellcolor[gray]{0.85}  & \cellcolor[gray]{0.85}  & \cellcolor[gray]{0}  & \cellcolor[gray]{0}& \cellcolor[gray]{0} &&& \\
&$16$         & \cellcolor[gray]{0.85}  & \cellcolor[gray]{0.85}  & $\bot$ & \cellcolor[gray]{0.85}  & \cellcolor[gray]{0.85}  & \cellcolor[gray]{0.85}  & $\bot$ & $\bot$ & \cellcolor[gray]{0}& \cellcolor[gray]{0} &&& \\
&$17$         & \cellcolor[gray]{0.85}  & \cellcolor[gray]{0.85}  & $\bot$ & $\bot$ & $\bot$ & $\bot$ & \cellcolor[gray]{0.85}  & \cellcolor[gray]{0.85}  & \cellcolor[gray]{0}  & \cellcolor[gray]{0}&&& \\
&$18$         & \cellcolor[gray]{0.85}  & \cellcolor[gray]{0.85}  & $\bot$ & \cellcolor[gray]{0.85}  & $\bot$ & $\bot$ & $\bot$ & ?      & $\bot$& \cellcolor[gray]{0} &&& \\
&$19$         & \cellcolor[gray]{0.85}  & \cellcolor[gray]{0.85}  & $\bot$ & $\bot$ & $\bot$ & $\bot$ & $\bot$ & $\bot$ & $\bot$& \cellcolor[gray]{0} &&& \\
&$20$         & \cellcolor[gray]{0.85}  & \cellcolor[gray]{0.85}  & $\bot$ & ?      & $\bot$ & ?      & $\bot$ & ?      & $\bot$& $\bot$ &&& \\
&$21$         & \cellcolor[gray]{0.85}  & \cellcolor[gray]{0.85}  & $\bot$ & $\bot$ & $\bot$ & $\bot$ & $\bot$      & ?      & ?& ?  &&&
\end{tabular}
\caption{Which accordions are PMH for $3\leq n\leq 21$ and $1\leq k\leq 10$}
\label{table acc pmh}
\end{table}

Additionally, as already remarked before, the main result in \cite{CpCqCirculant} gives more than just all the possible values of $n_{1}$ and $n_{2}$, for which $C_{n_{1}}\square C_{n_{2}}$ is circulant. In fact, the main result of the above paper is the following.

\begin{theorem}\cite{CpCqCirculant}
The circulant graph $\textrm{Ci}[n',\{a_{1},a_{2}\}]$ is isomorphic to $C_{n_{1}}\square C_{n_{2}}$ if and only if:
\begin{enumerate}[(i)]
\item $n'=n_1n_2$,
\item $n_1=\gcd(n',a_j)$ and $n_2=\gcd(n',a_{3-j})$, where $j=1$ or $j=2$, and
\item $\gcd(n_1,n_2)=1$.
\end{enumerate}
\end{theorem}

In this sense, we think that it would be an interesting endeavour to give a necessary and sufficient condition for a quartic circulant graph to be isomorphic to some accordion.

%\appendix
%\section{The PMH-property in small accordion graphs}\label{appendix}

\end{document}